\documentclass[11pt]{article}

\usepackage{mathrsfs}
\usepackage{amsfonts}
\usepackage{titlesec}
\usepackage{enumerate}
\usepackage[toc,page,title,titletoc,header]{appendix}
\usepackage{amsmath,amsthm,amssymb,anysize}
\usepackage[numbers,sort&compress]{natbib}
\usepackage{color}
\usepackage{graphicx}
\usepackage{threeparttable}
\usepackage[all]{xy}

\theoremstyle{definition}
\newtheorem{thm}{Theorem}[section]
\newtheorem{lem}{Lemma}[section]
\newtheorem{rem}[lem]{Remark}
\newtheorem{prop}[lem]{Proposition}
\newtheorem{defn}[lem]{Definition}
\newtheorem{cor}[lem]{Corollary}

\newtheorem{exa}{Example}
\newtheorem{remark}{Remark}
\newtheorem*{conclusion}{Conclusion}
\marginsize{3cm}{3cm}{0.5cm}{2.2cm}
\numberwithin{equation}{section}
\numberwithin{equation}{subsection}

\newcommand{\C}{\mathbb{C}}
\newcommand{\F}{\mathbb{F}}

\def\J{\mathfrak{J}}
\def\S{\mathrm{S}}

\title{\textsf{On cohomology and deformations of Jacobi-Jordan algebras}}
\author{Yong Yang $^{a,b}$
\\
\\ \small\textit{$^{a}$School of Mathematics and Statistics, Northeast Normal University, Changchun 130024, China
}
\\ \small\textit{$^{b}$Doctoral School of Physics, University of P\'{e}cs, P\'{e}cs 7622, Hungary
}
  }
\date{ }

\begin{document}
\makeatletter
\newcommand{\rmnum}[1]{\romannumeral #1}
\newcommand{\Rmnum}[1]{\expandafter\@slowromancap\romannumeral #1@}
\makeatother
\maketitle
\let\thefootnote\relax\footnote{
E-mail address: yangyong195888221@163.com}

\begin{quotation}
\small\noindent \textbf{Abstract}:
In this paper we study cohomology and  deformations of Jacobi-Jordan algebras. We develop their formal deformation theory. In particular, we introduce a method to construct a versal deformation for a given Jacobi-Jordan algebra, which can induce all deformations and is unique on the infinitesimal level.
We construct formal 1-parameter deformations of Jacobi-Jordan algebras up to dimension 5 and versal deformations for 3-dimensional Jacobi-Jordan algebras.

\vspace{0.2cm} \noindent{\textbf{Keywords}}: Jacobi-Jordan algebra; cohomology; formal deformation

\vspace{0.2cm} \noindent{\textbf{Mathematics Subject Classification 2020}}: 14B12, 17C10

\end{quotation}
\setcounter{section}{-1}
\section{Introduction}
Jacobi-Jordan algebras (JJ algebras for short) were introduced in \cite{D-A}  in and further studied \cite{A-M}. They are also called Jordan-Lie algebras \cite{J-L}, mock-Lie algebras \cite{Z} and Jordan-nilalgebras of index 3 \cite{J-C}. On one hand,  by definition, these algebras can be obtained from Lie algebras by replacing skew commutativity by commutativity. On the other hand, these algebras also can be viewed as a special class of Jordan algebras \cite{D-A}. However, they have a lot of  quite different properties from either Lie algebras or Jordan algebras, which raises many interesting questions. For example, it is proved that every JJ algebra is nilpotent \cite{D-A}. Recently, the cohomology and formal 1-parameter deformations of JJ algebras were developed in \cite{B5}.
The theory of JJ algebras has many applications not only in the study of modern mathematics
and theoretical physics \cite{K,Ne} but also in the study of Bernstein-Jordan algebras and train algebras \cite{A,S,C,FZ}. For example, JJ algebras are introduced as examples of the more popular and well-referenced Jordan algebras to achieve an axiomatization for the algebra of observables in quantum mechanics \cite{Ne}.

In this paper, we study  deformations of JJ algebras and we also construct deformations in low-dimensional cases.
Suppose that the ground field $\mathbb{F}$ is algebraically closed of
characteristic not 2 or 3. The structure of the paper is as follows: In Section 1, we introduce some related
definitions  and properties of JJ algebras, including symplectic and pseudo-euclidean structures and the  classification in low dimensions. In Section 2, we give the definition of  cohomology and Massey products. In Section 3, we introduce some basic definitions of (formal) deformations. In Section 4, we focus on the theory of  formal 1-parameter deformations and construct deformations up to dimension 5 by computing 2-cohomology. In Section 5, we study versal deformations. Starting with a universal infinitesimal deformation, one can extend it to a versal deformation step by step.
At last, we give the construction of versal deformation in dimension 3.

\section{Jacobi-Jordan algebras}
\begin{defn}\cite[Definition 2.1]{D-A}
An algebra $(\mathfrak{J}, \cdot)$ is called a \emph{Jacobi-Jordan algebra} (\emph{JJ algebra} for short), if it satisfies the following two identities:
\begin{eqnarray*}
  &&x\cdot y= y\cdot x \\
  &&x\cdot(y\cdot z)+y\cdot(z\cdot x)+z\cdot(x\cdot y)=0\ \text{(Jacobi identity)}
\end{eqnarray*}
for all $x,y,z\in \mathfrak{J}$.
\end{defn}

\begin{exa}\cite{D-A}
Let $H_m$ be a space with a basis
$$\{x_1,\ldots, x_m,y_1,\ldots,y_m,z\}.$$
The JJ algebra structure is given by
$x_iy_i=y_ix_i=z$ for $i=1,\ldots,m$ and is called \emph{Heisenberg JJ algebra}.
\end{exa}

\begin{defn}\cite[Proposition 4.1]{B4}
Suppose that $(\mathfrak{J}, \cdot)$ is  a JJ algebra, $V$ is a  vector space and $\pi: \J\longrightarrow \mathrm{End}(V)$ is a linear mapping. If $\pi$ satisfies
$$\pi(x)\pi(y)+\pi(y)\pi(x)=-\pi(x\cdot y),$$
for any $x,y \in\J$, then $(\pi, V)$ is called a \emph{representation} of $(\mathfrak{J}, \cdot)$.
\end{defn}

\begin{exa}
A JJ algebra $\J$ is a representation of itself with respect to the multiplication  of $\J$. In this case, it is called the \emph{adjoint representation} of $\J$.
\end{exa}
\subsection{Classification in low dimensions}
Classification  of low-dimensional JJ algebras over algebraically closed fields of characteristic not 2 or 3 has been considered in \cite{D-A}. From now on, we suppose that $\mathfrak{J}_{m}$ is an $m$-dimensional JJ algebra, spanned by
$$\{e_{1},\cdots,e_m\}.$$
We recall the classification of nontrivial JJ algebras in \cite{D-A}:

$\bullet$ in dimension 2,

(1) $\mathfrak{J}_{1,2}$: $ e_1^{2}=e_2$,

$\bullet$ in dimension 3,

(1) $\mathfrak{J}_{1,2}\oplus \F$: $ e_1^{2}=e_2$,

(2) $\mathfrak{J}_{1,3}$: $e_1^{2}=e_3^{2}=e_2$,

$\bullet$ in dimension 4,

(1) $\mathfrak{J}_{1,2}\oplus \F^{2}$: $ e_1^{2}=e_2$,

(2) $\mathfrak{J}_{1,3}\oplus \F$: $e_1^{2}=e_3^{2}=e_2$,

(3) $\mathfrak{J}_{1,2}^{2}$: $e_1^{2}=e_2$, $e_3^{2}=e_4$,

(4) $\mathfrak{J}_{1,4}$: $e_1^{2}=e_2$, $e_1\cdot e_3=e_4$,

(5) $\mathfrak{J}_{2,4}$: $e_1^{2}=e_3\cdot e_4=e_2$,

$\bullet$ In dimension 5,

(1) $\mathfrak{J}_{1,2}\oplus \F^{3}$: $ e_1^{2}=e_2$,

(2) $\mathfrak{J}_{1,3}\oplus \F^{2}$: $e_1^{2}=e_3^{2}=e_2$,

(3) $\mathfrak{J}_{1,2}^{2}\oplus \F$: $e_1^{2}=e_2$, $e_3^{2}=e_4$,

(4) $\mathfrak{J}_{1,4}\oplus \F$: $e_1^{2}=e_2$, $e_1\cdot e_3=e_4$,

(5) $\mathfrak{J}_{2,4}\oplus \F$: $e_1^{2}=e_3\cdot e_4=e_2$,

(6) $\mathfrak{J}_{1,2}\oplus \mathfrak{J}_{1,3}$: $e_1^{2}=e_2$, $e_3^{2}=e_5^{2}=e_4$,

(7) $\mathfrak{J}_{1,5}$: $e_1^{2}=e_2$, $e_1\cdot e_3=e_5$, $e_3^{2}=e_4$,

(8) $\mathfrak{J}_{2,5}$: $e_1^{2}=e_2$, $e_1\cdot e_4=e_3^{2}=e_5$,

(9) $\mathfrak{J}_{3,5}$: $e_1^{2}=e_2$, $e_1\cdot e_4=e_3\cdot e_4=e_5$, $e_3^{2}=-e_2+e_5$,

(10) $\mathfrak{J}_{4,5}$: $e_1^{2}=e_2$, $e_3^{2}=e_4$, $e_5^{2}=-e_2+e_4$,

(11) $\mathfrak{J}_{5,5}$: $e_1^{2}=e_2$, $e_3^{2}=e_4$, $e_3\cdot e_5=-e_2+e_4$,

(12) $\mathfrak{J}_{6,5}$: $e_1^{2}=e_1\cdot e_4=e_2$, $e_1\cdot e_3=e_5$,

(13) $\mathfrak{J}_{7,5}$: $e_1^{2}=e_3^{2}=e_4\cdot e_5=e_2$,

(14) $\mathfrak{J}_{8,5}$: $e_1^{2}=e_2$, $e_1\cdot e_4=e_5$, $e_2\cdot e_4=-2e_1\cdot e_5=2e_3$.
\begin{remark}
The classification has stopped at dimension 5, since in dimension 6 infinitely many isomorphism classes of commutative, associative algebras show up \cite{C1}.
\end{remark}
\begin{remark}
An algebra $(\mathfrak{L}, [\ ,\ ])$ is called a \emph{(left) Leibniz algebra} if it satisfies
$$[x,[y,z]]=[[x,y],z]+[y,[x,z]]$$
for any $x,y,z \in\mathfrak{L}$.
Note that except for $\mathfrak{J}_{8,5}$, every JJ algebra $\J$ of dimension $\leq 5$ is associative and $\J^{3}=0$. We obtain that except for $\mathfrak{J}_{8,5}$, every JJ algebra in the classification is Leibniz algebra as well.
\end{remark}

\subsection{Symplectic and  pseudo-euclidean structures}

The theory of symplectic and  pseudo-euclidean structures were introduced for Jordan algebras in \cite{B1,B2,B3} and developed for JJ algebras in \cite{B4}. Here we study the symplectic and  pseudo-euclidean structures of low-dimensional JJ algebras.
\begin{defn}
A JJ algebra $(\mathfrak{J},\cdot)$ is called

(1) \emph{symplectic} if it is endowed with a non-degenerate skew-symmetric bilinear form $w$ satisfying
$$\omega(x\cdot y,z)+\omega(y\cdot z,x)+\omega(z\cdot x,y)=0,$$
for any $x,y,z\in \mathfrak{J}$;

(2) \emph{pseudo-euclidean} if it is endowed with a non-degenerate  symmetric bilinear form $B$ satisfying
$$B(x\cdot y,z)=B(x,y\cdot z),$$
for any $x,y,z\in \mathfrak{J}$.
\end{defn}

\begin{cor}
The dimension of a symplectic JJ algebra is even.
\end{cor}

\begin{defn}
 Two symplectic (or pseudo-euclidean) JJ algebras $(\J,w)$ and $(\J' ,w')$
are called \emph{isometrically isomorphic} (or \emph{i-isomorphic}) if there exists a JJ  algebra
isomorphism $f: \J \rightarrow \J'$ satisfying $w(f(x),f(y))=w'(x,y)$ for all $x,y\in\J$.
\end{defn}

\begin{defn}
Let $(\mathfrak{J},\omega)$ be a symplectic (or pseudo-euclidean) JJ algebra. An ideal $I$ of $\mathfrak{J}$ is called non-degenerate if $w|_{I\times I}$ is non-degenerate.
\end{defn}

Similar to \cite[Proposition 2.6]{M1}, the following Lemma allows us to focus our studies on indecomposable cases.
\begin{lem}
 Let $(\mathfrak{J},\omega)$ be a symplectic (or pseudo-euclidean) JJ algebra. Then
 $$\mathfrak{J}=\bigoplus\limits_{i=1}^{r}\mathfrak{J}_{i},$$
such that, for all $1\leq i\leq r$,

(1) $\mathfrak{J}_{i}$ is a non-degenerate ideal of $\mathfrak{J}$.

(2) $\mathfrak{J}_{i}$ contains no nontrivial non-degenerate ideal of $\mathfrak{J}$.

(3) $\omega(\mathfrak{J}_{i},\mathfrak{J}_{j})=0$ for all $i\neq j$.
\end{lem}
\begin{defn}
In the above Lemma, if $r=1$, $\J$ is called \emph{indecomposable (or irreducible)}.  Otherwise, $\J$ is called \emph{decomposable (or reducible)}.
\end{defn}
As in Lie \cite{Lie} and Lie superalgebra cases \cite{M1}, the theory of double extensions is important tool to classify symplectic (or pseudo-euclidean) JJ algebras inductively. Now we introduce the theory of symplectic double extensions. For more details, the reader is referred to \cite{B4}.

\begin{defn}\cite[Definition 2.4 (2)]{B4}
Let $(\J,\cdot)$ be a JJ algebra and $D$ a linear transformation of $\J$. $D$ is called an \emph{anti-derivation} of $\J$ if
$$D(x\cdot y)=-D(x)\cdot y-x\cdot D(y),\ x,y\in\J.$$
Denote $\mathrm{Ader}(\J)$ the anti-derivation space of $\J$.
\end{defn}
\begin{remark}
 For $x\in\J$, define $L_x:\J\longrightarrow\J$ by $L_x(y)=x\cdot y$ for any $y\in\J$. Then $L_x\in \mathrm{Ader}(\J)$.
\end{remark}

\begin{defn}\cite[Definition 3.2]{B4}
A \emph{special admissible pair} of a symplectic JJ algebra $(\J,\omega)$ is a pair $(D,A_0)$, where $D\in \mathrm{Ader}(\J)$ and $A_0\in \mathrm{Ker}\ D$ such that $D^{2}=-\frac{1}{2}L_{A_{0}}$ and $w(A_0,\mathrm{Im}\ D)=0$.
\end{defn}
\begin{defn}
Suppose that $(\J,\omega)$ is a symplectic JJ algebra. Let $f$ be a linear transformation of $\J$. Then the linear transformation $f^{\ast}$ of $\J$, defined by
$$\omega(f(x),y)=\omega(x,f^{\ast}(y)),\quad \forall x,y\in \J,$$
is called the \emph{adjoint} of $f$ with respect to $\omega$.
\end{defn}

\begin{thm}\cite[Theorem 3.1]{B4}\label{d1}
Let $(\J,\cdot, \omega)$ be a symplectic JJ algebra and $(D,A_0)$ a special admissible pair of $\J$. Then the space $\widetilde{\J}=\F e\oplus\J\oplus\F e^{\ast}$, endowed with the products
$$e^{2}=A_0,\ e\star x=D(x)+\frac{1}{2}\omega(A_0,x)e^{\ast},\ x\star y=x\cdot y+\omega((D-D^{\ast})(x),y)e^{\ast},$$
 where $D^{\ast}$ is the adjoint of $D$ with respect to $\omega$ and $x,y\in\J$,
  is a JJ algebra. Moreover, the skew-symmetric bilinear form $\widetilde{\omega}$ defined by $$\widetilde{\omega}|_{\J\times\J}=\omega,\ \widetilde{\omega}(e,e^{\ast})=1,$$
 is a symplectic form on $(\widetilde{\J},\star)$. The symplectic JJ algebra $(\widetilde{\J},\star, \widetilde{\omega})$ is called a symplectic double extension of
$(\J,\cdot, \omega)$ (by means of $(D,A_0)$).
\end{thm}

\begin{thm}\cite[Theorem 3.3]{B4}\label{d2}
Let $(\widetilde{\J},\star, \widetilde{\omega})$ be a symplectic JJ algebra such that $\widetilde{\J}\neq 0$. Then $(\widetilde{\J},\star, \widetilde{\omega})$ is a symplectic double extension of a symplectic JJ algebra $(\J,\cdot, \omega)$.
\end{thm}

\begin{thm}
Over an algebraically closed field of characteristic not 2 or 3, the classification of nontrivial indecomposable symplectic and pseudo-euclidean JJ algebras up to dimension 5 is as follows:

(1) there is only one nontrivial  indecomposable symplectic JJ algebra: $\J_{1,4}$. A symplectic form can be given by
$$\omega_{\J_{1,4}}=\left(
  \begin{array}{cccc}
    0 & 0 & 0 & 1 \\
    0 & 0 & 2 & 0 \\
    0 & -2 & 0 & 0 \\
    -1 & 0& 0 & 0 \\
  \end{array}
\right);$$

(2) there are  two nontrivial indecomposable  pseudo-euclidean JJ algebras: $\J_{1,2}$, $\J_{1,4}$.
Pseudo-euclidean forms can be  given by
$$B_{\J_{1,2}}=\left(
  \begin{array}{cc}
    0 & 1 \\
    1 & 0 \\
  \end{array}
\right),\
B_{\J_{1,4}}=\left(
  \begin{array}{cccc}
    0 & 0 & 0 & 1 \\
    0 & 0 & 1 & 0 \\
     0& 1 & 0 & 0 \\
    1 & 0 & 0 & 0 \\
  \end{array}
\right).$$
\end{thm}
\begin{proof}
A direct proof can be given by checking the above classification. Here we give a proof by the method of double extensions. As an example, we classify symplectic JJ algebras up to dimension 5. Note that the dimension of  a symplectic JJ algebra is even. It is sufficient to compute double extensions of 0-dimensional and 2-dimensional symplectic JJ algebras respectively. By Theorems \ref{d1} and \ref{d2}, we obtain that there is only the trivial symplectic JJ algebra $\F^{2}$ in dimension 2. Suppose that $\F^{2}=(\mathrm{span}\{e_1,e_2\},\omega_{\F^{2}})$ is the 2-dimensional symplectic JJ algebra with the trivial multiplication, where
$\omega_{\F^{2}}=\left(
                   \begin{array}{cc}
                     0 & 1 \\
                     -1 & 0 \\
                   \end{array}
                 \right)$.
Let $(D,A_0)$ be a special admissible pair of $\F^{2}$ and $\widetilde{\J}=\F e\oplus\J\oplus\F e^{\ast}$ be a double extension of $(\F^{2},\omega_{\F^{2}})$ by means of $(D,A_0)$. Then we have $D^{2}=0$.
It is sufficient to prove that all double extensions of $(\F^{2},\omega_{\F^{2}})$ by means of $(D,A_0)$ are isomorphic to $\J_{1,4}$ in the following cases.

Case 1. $D=0$. Then $D^{\ast}=0$. Set $A_0=ke_1+le_2$, where $k,l\in\F$. If $k=l=0$, then $\widetilde{\J}$ is the 4-dimensional trivial JJ algebra. If $k\neq 0$, then an isomorphism to $\J_{1,4}$ can be given by the change of basis:
$$e'_{1}=\frac{1}{\sqrt{k}}e,\ e'_2=e_1+\frac{l}{k}e_2,\ e'_3=e_2,\ e'_4=\frac{\sqrt{k}}{2}e^{\ast}.$$
Otherwise, we have $k=0,l\neq 0$. In this case, an isomorphism can be given by the change of basis:
$$e'_1=e,\ e'_2=le_2,\ e'_3=e_1,\ e'_4=-\frac{l}{2}e^{\ast}.$$

Case 2. $D\neq 0$. Set $D(e_1)=a^{1}_{1}e_1+a^{1}_{2}e_2$, $D(e_2)=a^{2}_{1}e_1+a^{2}_{2}e_2$ and $A_0=ke_1+le_2$. If $a^{1}_{2}=a^{2}_{1}=0$, we get $D=0$ from $D^{2}=0$. Without loss of generality, we can suppose $a^{1}_{2}\neq 0$. By Theorem \ref{d1}, we get the products:
$$e^{2}=\frac{la^{1}_{1}}{a^{1}_{2}}e_1+le_2,\ e\star e_{1}=a^{1}_{1}e_1+a^{1}_{2}e_2-\frac{l}{2}e^{\ast},\ e\star e_{2}=-\frac{(a^{1}_{1})^{2}}{a^{1}_{2}}e_1-a^{1}_{1}e_2+\frac{la^{1}_{1}}{2a^{1}_{2}}e^{\ast},$$
$$e_{1}^{2}=-2a^{1}_{2}e^{\ast},\ e_{1}\star e_{2}=2a^{1}_{1}e^{\ast},\ e_{2}^{2}=-\frac{2(a^{1}_{1})^{2}}{a^{1}_{2}}e^{\ast}.$$
Moreover, an isomorphism to $\J_{1,4}$ can be given by the change of basis:
$$e'_1=e_1,\ e'_2=-2a^{1}_{2}e^{\ast},\ e'_3=e-\frac{l}{2a^{1}_{2}}e_1,\ e'_4=a^{1}_{1}e_{1}+a^{1}_{2}e_2+\frac{l}{2}e^{\ast}.$$
\end{proof}

\begin{remark}
As Leibniz algebras, the pseudo-euclidean structures of $\J_{1,2}$ and $\J_{1,4}$ are also given in \cite{M2}.
\end{remark}
\begin{remark}
As examples, symplectic and pseudo-euclidean structures of 4-dimensional JJ algebras are studied in \cite{B4}, but there is a mistake of symplectic form of $\J_{1,4}$ \cite[Example 3.1]{B4}.
In particular, it is proved that $\J_{1,4}$ is the unique nontrivial JJ algebra in dimension 4 which both has a pseudo-euclidean  and symplectic structure \cite[Remark 4.1]{B4}. From our result, we obtain that $\J_{1,4}$ is the unique nontrivial pseudo-euclidean  and symplectic JJ algebra up to dimension 5.
\end{remark}
\begin{remark}
Recall a  Leibniz algebra is called \emph{metric} if there is a pseudo-euclidean form on it \cite[Definition 2.5]{M2}. In the classification, the Leibniz 2-cohomology and metric Leibniz deformations of $\J_{1,2}$, $\J_{1,2}\oplus \C$, $\J_{1,2}\oplus \C^{2}$, $\J_{1,2}^{2}$, $\J_{1,4}$,
$\J_{1,2}\oplus \C^{3}$, $\J_{1,2}^{2}\oplus \C$ and $\J_{1,4}\oplus \C$ have been studied over the complex field $\C$ \cite{M2}. We recall the results.
$$
\begin{array}{c|c|c|c|c|c|c|c|c}
\hline
  \mathrm{Dimension} & 2 & 3 & 4 & 4 & 4 & 5 & 5 & 5\\
\hline
  \mathrm{Algebra} & \mathfrak{J}_{1,2} & \mathfrak{J}_{1,2}\oplus \C & \mathfrak{J}_{1,2}\oplus\C^{2} & \mathfrak{J}_{1,2}^{2} & \mathfrak{J}_{1,4} & \mathfrak{J}_{1,2}\oplus\C^{3} & \mathfrak{J}_{1,2}^{2}\oplus\C & \mathfrak{J}_{1,4}\oplus\C\\
\hline
  \mathrm{dim}\ \mathrm{HL}^{2} & 2 & 8 & 27 & 8 & 9 & 64 & 29 &30 \\
\hline
\end{array}
$$
There are only trivial metric Leibniz deformations in dimension 2 and 3. The pictures of metric Leibniz deformations in dimension 4 and 5 are as follows:

in dimension 4,
$$\xymatrix{
       \J_{1,2}\oplus\C^{2} \ar[d]^{}          &             \\
   \J_{1,4} \ar[d]^{}  & \\
\J_{1,2}^{2}  & }$$

in dimension 5,
$$\xymatrix{
       \J_{1,2}\oplus\C^{3} \ar[d]^{}          &             \\
   \J_{1,4}\oplus\C \ar[d]^{}  & \\
\J_{1,2}^{2}\oplus\C   & }$$
where the down arrows show deformations. In this paper, we compute JJ deformations  for JJ algebras of dimension $\leq 5$ by computing JJ 2-cohomology with coefficients in adjoint modules.
\end{remark}
\section{Cohomology and Massey products}

Here we give the definition of 2-cohomology of JJ algebras following the Lie algebra case described by Fialowski \cite{exma2}.
Since we only need the cohomology with coefficients in the adjoint representation,  we restrict our definition to the adjoint cohomology case.

Suppose that $(\mathfrak{J}, \cdot)$ is a JJ algebra. Define the $n$-cochain space $\S^{n}(\mathfrak{J},\mathfrak{J})$ by the space of symmetric $n$-linear mappings $\mathfrak{J}\times\cdots\times \mathfrak{J}\rightarrow \mathfrak{J}$.
 We define the complex
$$\S^{1}(\mathfrak{J},\mathfrak{J})\stackrel{\mathrm{d}}{\longrightarrow} \S^{2}(\mathfrak{J},\mathfrak{J})\stackrel{\mathrm{d}}{\longrightarrow} \S^{3}(\mathfrak{J},\mathfrak{J}),$$
where the differential is defined by
\begin{eqnarray*}
&\mathrm{d}\varphi(x,y)=\varphi(x\cdot y)-x\cdot\varphi(y)-y\cdot\varphi(x), \\
&\mathrm{d}\varphi(x,y,z)=\varphi(x,y\cdot z)+\varphi(y,z\cdot x)+\varphi(z,x\cdot y)+x\cdot\varphi(y,z)+y\cdot\varphi(z,x)+z\cdot\varphi(x,y).
\end{eqnarray*}
 In order to describe deformations of JJ algebras, we define  \emph{2-cohomology} by
$\mathrm{H}^{2}(\mathfrak{J},\mathfrak{J})=\mathrm{Z}^{2}(\mathfrak{J},\mathfrak{J})/\mathrm{B}^{2}(\mathfrak{J},\mathfrak{J})$,
where
\begin{eqnarray*}
&& \mathrm{Z}^{2}(\mathfrak{J},\mathfrak{J})=\mathrm{Ker}\ (\mathrm{d}: \S^{2}(\mathfrak{J},\mathfrak{J})\rightarrow \S^{3}(\mathfrak{J},\mathfrak{J})), \\
 &&\mathrm{B}^{2}(\mathfrak{J},\mathfrak{J})=\mathrm{Im}\ (\mathrm{d}: \S^{1}(\mathfrak{J},\mathfrak{J})\rightarrow \S^{2}(\mathfrak{J},\mathfrak{J})).
\end{eqnarray*}
A cochain is called a \emph{cocycle} (resp. \emph{coboundary}) if it is in $\mathrm{Z}^{2}(\mathfrak{J},\mathfrak{J})$
 $(\mathrm{resp}.\ \mathrm{B}^{2}(\mathfrak{J},\mathfrak{J}))$. From now on, we write $\mathbb{H}$ by $\mathrm{H}^{2}(\J,\J)$ for a JJ algebra $\J$.

For $\varphi\in \S^{p}(\mathfrak{J},\mathfrak{J}),\psi\in \S^{q}(\mathfrak{J},\mathfrak{J})$, we define $\varphi\psi\in \S^{p+q-1}(\mathfrak{J},\mathfrak{J})$ by
\begin{eqnarray*}
(\varphi\psi)(x_{1},\ldots,x_{p+q-1}) &=& \sum\limits_{1\leq i_1<\cdots<i_{p-1}\leq p+q-1}\varphi(x_{i_1},\ldots,x_{i_{p-1}}, \\
   &&\psi(x_1,\ldots,\widehat{x}_{i_{1}},
\ldots,\widehat{x}_{i_{p-1}},\ldots,x_{p+q-1})).
\end{eqnarray*}
Then the \emph{Massey 2-product or Massey square} of $\varphi$ and $\psi$ is defined by $[\varphi,\psi]=\varphi\psi-(-1)^{(p-1)(q-1)}\psi\varphi$. If we define  $\mathcal{S}^{p}=\mathrm{S}^{p+1}(\J,\J)$, then the Massey products make $\mathcal{S}^{\bullet}=\bigoplus_{p}\mathcal{S}^{p}$ a $\mathbb{Z}$-graded algebra.
For a cochain $\varphi$, we denote the $\mathbb{Z}$-degree of $\varphi$ by $\|\varphi\|$.
By a direct computation, we have the following Theorem.
\begin{thm}\label{a}
For $\varphi,\psi,\omega\in\mathcal{S}^{\bullet}$,

(1) $[\varphi,\psi]=-(-1)^{\|\varphi\|\|\psi\|}[\psi,\varphi];$

(2) $\circlearrowleft_{\varphi,\psi,\omega}(-1)^{\|\varphi\|\|\omega\|}[\varphi,[\psi,\omega]]=0$;

(3) if $\|\varphi\|=1$, then $\mathrm{d}\varphi=[\varphi_0,\varphi]$, where $\varphi_0$ is defined by the multiplication of $\mathfrak{J}$.
\end{thm}

The Massey squares also can be defined on $\mathrm{H}^\bullet(\J,\J)$ by
$[\overline{\varphi}_1, \overline{\varphi}_2]=\overline{[\varphi_1, \varphi_2]}$ for $\varphi_1, \varphi_2\in \mathrm{Z}^{2}(\J,\J)$.
These Massey squares are in $\mathrm{H}^3(\J,\J)$.
As in Lie case \cite{exma2}, higher order
the Massey products on cohomology spaces are also defined, and those products are also elements of $\mathrm{H}^3$.
They are only defined if all the participating lower order Massey products are trivial. Let us define the third order Massey operation $<\overline{\varphi}_1, \overline{\varphi}_2, \overline{\varphi}_3>$  when the brackets $[\overline{\varphi}_i,\overline{\varphi}_j]$ are all trivial for $i,j \in \{1,2,3\}$ so that there exist cocycles $\varphi_{ij}$ such that $[\varphi_i,\varphi_j]=\mathrm{d}\varphi_{ij}$. The \emph{third order Massey operation (cube)} takes value in the factor space
$$
\mathrm{H}^3(\J,\J)/[\overline{\varphi}_1, \mathrm{H}^2(\J,\J)]+[\overline{\varphi}_2, \mathrm{H}^2(\J,\J)]+[\overline{\varphi}_3, \mathrm{H}^2(\J,\J)],
$$
and is equal to the image of the cohomology class of the cocycle
$$
[\varphi_{12}, \varphi_3]+[\varphi_{23},\varphi_1]+[\varphi_{13}, \varphi_2].
$$
The image of this cohomology class in the factorspace is well-defined.
For higher order Massey operations see \cite{exma2}.

\section{General theory of (formal) deformations}

Here we adopt the general deformation theory for Lie algebras, worked out by Fialowski \cite{F, exma2}, to the category of JJ algebras.
Let $A$ be a commutative algebra over $\F$ with identity, and $\mathfrak{m}_{A}$ a maximal ideal of $A$. We denote by $\varepsilon_{A}: A\rightarrow \F$ the  augmentation
of $A$ according to $\mathfrak{m}_{A}$, satisfying $\varepsilon_{A}(1)=1$  and $\mathrm{Ker}\ \varepsilon_{A}=\mathfrak{m}_{A}$.
\begin{defn}\label{e}
A \emph{deformation} $\lambda$ of a JJ algebra $\J$ with base $(A,\mathfrak{m}_{A})$ is
 a family of JJ algebra structures on the $A$-module $\mathfrak{J}_A=A\otimes_{\F}\mathfrak{J}$ such that
$$\varepsilon\otimes \mathrm{id}_{\J}:\J_{A}\rightarrow \F\otimes\J=\J$$
is a JJ algebra homomorphism.
\end{defn}
\begin{defn}
Two deformations $\lambda$ and $\lambda'$ of a JJ algebra $\J$ with the same base $A$ are called \emph{equivalent} if there exists a JJ algebra isomorphism $\varphi$ such that the following diagram commutes:
$$\xymatrix{
  (\J_A,\lambda) \ar[dr]_{\varepsilon_{A}\otimes \mathrm{id}_{\J}} \ar[r]^{\varphi}
                &  (\J_A,\lambda')\ar[d]^{\varepsilon_{A}\otimes \mathrm{id}_{\J}}  \\
                &\F\otimes\J            }
$$
\end{defn}

\begin{defn}
A deformation with base $A$ is called \emph{local}, if $A$ is local. It is called of order $k$, if $\mathfrak{m}_{A}^{k+1}=0$. In particular, a local deformation of order 1 is called \emph{infinitesimal}.
\end{defn}

Suppose that $A'$ is another commutative algebra with identity over $\F$ with the corresponding augmentation $\varepsilon_{A'}:A'\rightarrow\F$. Let $\varphi:A\rightarrow A'$ be an algebra homomorphism with $\varphi(1)=1$ and $\varepsilon_{A'}\circ \varphi=\varepsilon_{A}$.
An $A$-module structure on $A'$ can be given by
$$aa'=a'\varphi(a),\quad a\in A,a'\in A'.$$
We have
$$A'\otimes_{\F} \J=(A'\otimes_{A}A)\otimes_{\F} \J=A'\otimes_{A}(A\otimes_{\F} \J)=A'\otimes_{A}\J_A.$$
A deformation $\lambda'$ of $\J$ with base $(A',\mathfrak{m}_{A'})$ is given by
$$((a'_1\otimes_A(a_1\otimes x_1))\cdot (a'_2\otimes_A(a_2\otimes x_2)))_{\lambda'}=a'_1a'_2\otimes_{A}((a_1\otimes x_1)\cdot(a_2\otimes x_2))_{\lambda},$$
where $a_1,a_2\in A$, $a'_1,a'_2\in A'$ and $x_1,x_2\in \J$.
\begin{defn}
With the above notations, the deformation $\lambda'$ is called the \emph{push-out} of the deformation $\lambda$ by means of $\varphi$, denoted by $\varphi_{\ast}\lambda$.
\end{defn}
\begin{rem}
If the deformation $\lambda$ is given by
$$((1\otimes x_1)\cdot(1\otimes x_2))_{\lambda}=1\otimes x_1x_2+\sum_{i}m_i\otimes y_i,\quad m_i\in\mathfrak{m}_A, y_i\in\J,$$
then, by definition, the push-out $\varphi_{\ast}\lambda$ is given by
\begin{equation}\label{po}
((1\otimes x_1)\cdot(1\otimes x_2))_{\varphi_{\ast}\lambda}=1\otimes x_1x_2+\sum_{i}\varphi(m_i)\otimes y_i.
\end{equation}
\end{rem}

Suppose that $A$ is a local algebra with maximal ideal $\mathfrak{m}_{A}$. The \emph{inverse (or projective) limit} of $A$ is defined by
$$\mathop{\overleftarrow{\mathrm{lim}}}\limits_{n\rightarrow\infty}(A/\mathfrak{m}_{A}^{n})=\left\{\{a_i\}\in\prod^{\infty}_{n=1} (A/\mathfrak{m}_{A}^{n})\mid a_{n+1}\equiv a_{_{n}}\ \mathrm{\mathrm{mod}}\ \mathfrak{m}_{A}^{n}, n\geq 1\right\}.$$
The algebra $A$ is called \emph{complete}, if $A=\mathop{\overleftarrow{\mathrm{lim}}}\limits_{n\rightarrow\infty}(A/\mathfrak{m}_{A}^{n})$.
\begin{defn}
Suppose that $A$ is a complete local algebra with the maximal ideal $\mathfrak{m}_{A}$. A \emph{formal deformation} $\lambda$ of a JJ algebra $\J$ with base $(A,\mathfrak{m}_{A})$ is
 a family of JJ algebra structures on the completed tensor product $A\hat{\otimes} \J =\mathop{\overleftarrow{\mathrm{lim}}}\limits_{n\rightarrow\infty}((A/\mathfrak{m}_{A}^{n})\otimes\J)$ such that
$$\varepsilon_{A}\hat{\otimes} \mathrm{id}_{\J}:A\hat{\otimes} \J\rightarrow \F\otimes\J=\J$$
is a JJ algebra homomorphism.
\end{defn}
\begin{remark}
If $A=\mathbb{F}[\![t]\!]$, then a formal deformation of $\mathfrak{J}$ with base $A$ is the same as a formal 1-parameter deformation of $\mathfrak{J}$ (see \cite{exma2}).
\end{remark}

Since any complete local algebra can be represented as a quotient algebra of $\mathbb{F}[\![t_1,\ldots,t_n]\!]$ of formal power series, we consider formal deformations with the base $\mathbb{F}[\![t_1,\ldots,t_n]\!]$ in the next two sections, including $n=1$ (Section 4) and $n\geq 2$ (Section 5). Note that the above definitions and the theory is extended to these two cases in an obvious way.

\section{Formal 1-parameter deformations}

\begin{defn} Generalization of the Lie algebra case, see \cite{exma2}.
A \emph{formal 1-parameter deformation} of a JJ algebra $(\mathfrak{J}, \cdot)$ is a family of JJ algebra structures on the $\mathbb{F}[\![t]\!]$-module $\mathfrak{J}[\![t]\!]=\mathbb{F}[\![t]\!]\otimes_{\F}\mathfrak{J}$ such that
$$(x\cdot y)_{t}=\sum_{i=0}^{\infty} t^{i} \varphi_{i}(x,y),\ x,y\in \mathfrak{J}$$,
where each $\varphi_{i}$ is in $\S^{2}(\mathfrak{\mathfrak{J}}, \mathfrak{J})$ and $\varphi_{0}(x,y)=x\cdot y$.
\end{defn}
\begin{rem}
Since the JJ algebra structure on $\mathfrak{J}[\![t]\!]$ is uniquely determined by the sequence $\{\varphi_n\}_{n=1}^{\infty}$, we also call the sequence $\{\varphi_n\}_{n=1}^{\infty}$ the formal 1-parameter deformation of $\mathfrak{J}$.
\end{rem}
\begin{defn}
A formal 1-parameter deformation is called \emph{jump deformation} if for any non-zero value of the parameter $t$ near the origin, it gives isomorphic algebras (which is of course different from the original one), see Gerstenhaber \cite{G}.
\end{defn}
\begin{rem}
Formal deformations are transitive: If $A$ deforms to $B$ and $B$ deforms to $C$, then $A$ deforms to $C$ as well (see \cite{FP}).
\end{rem}
\begin{defn}(following \cite{exma2})
Suppose that $\{\varphi_n\}_{n=1}^{\infty}$ and $\{\varphi'_n\}_{n=1}^{\infty}$ are two formal 1-parameter deformations of a JJ algebra $\mathfrak{J}$. They are called \emph{equivalent} if there exists a linear isomorphism $\widehat{\psi}_{t}=\mathrm{id}_{A}+\psi_1 t+\psi_2 t^{2}+\cdots$, where $\psi_i$ is in $\S^{2}(\mathfrak{J},\mathfrak{J})$, such that
$$\widehat{\psi}_{t}(x \cdot y)_{t}=(\widehat{\psi}_{t}(x)\cdot\widehat{\psi}_{t}(y))'_t,\quad \mathrm{for}\ x,y \in \mathfrak{J}.$$
\end{defn}
\begin{defn}
A formal 1-parameter deformation is called \emph{trivial} if it is equivalent to the original algebra. If every formal 1-parameter deformation  of a JJ algebra $\mathfrak{J}$ is trivial, then $\mathfrak{J}$ is called \emph{rigid}.
\end{defn}
\begin{rem}
An infinitesimal deformation defines a \emph{real} deformation if there are no higher order $t$-terms in the Jacobi identity, which means it can be extended to a formal deformation without obstructions.
\end{rem}

\begin{thm}\label{b}
The sequence $\{\varphi_n\}_{n=1}^{\infty}$ defines a formal 1-parameter deformation of a JJ algebra $\mathfrak{J}$ if and only if the elements $\varphi_n$ in $\S^{2}(\mathfrak{J},\mathfrak{J})$ satisfy the equations
$$\mathrm{d}\varphi_n+\frac{1}{2}\sum_{\mbox{\tiny$\begin{array}{c}
i+j=n\\
i,j>0\\
\end{array}$}}[\varphi_i,\varphi_j]=0,\ n\geq 1.$$
\end{thm}
\begin{proof}
For $n\geq 1$, the coefficient of $t^{n}$ in the Jacobi identity for $x,y,z\in \mathfrak{J}$ equals to
$$\sum_{\mbox{\tiny$\begin{array}{c}
\circlearrowleft x,y,z\\
i+j=n,i,j\geq0,\\
\end{array}$}}
\varphi_i(x,\varphi_j(y,z))=\frac{1}{2}
\sum_{\mbox{\tiny$\begin{array}{c}
i+j=n\\
i,j\geq0\\
\end{array}$}}[\varphi_i,\varphi_j](x,y,z)=0.$$
From Theorem \ref{a} (3), the proof is complete.
\end{proof}
\begin{cor}
For $n=1$ in the above Theorem, we get that $\varphi_1$ is a cocycle. More generally, if $\varphi_{1}=0$, the first non-vanishing cochain $\varphi_{i}$ will be a cocycle.
\end{cor}
\begin{cor}
Suppose $\varphi$ is a cocycle. Then $\varphi$ defines a real deformation if and only if $[\varphi,\varphi]=0$.
\end{cor}
\begin{proof}
Note that $\varphi$ defines a real deformation if and only if
the sequence $\{\varphi_n\}_{n=1}^{\infty}$ defines a formal deformation, where $\varphi_1=\varphi$, $\varphi_2=\cdots=\varphi_n=\cdots=0$. The proof follows from Theorem \ref{b}.
\end{proof}
\begin{prop}\cite{exma2}
Every equivalence class of formal deformations defines uniquely a 2-cohomology class.
\end{prop}

\begin{cor}
If $\mathbb{H}=0$, then $\mathfrak{J}$ is rigid.
\end{cor}
\begin{rem}
The condition $\mathbb{H}=0$ is sufficient for $\mathfrak{\mathfrak{J}}$ being rigid, but not necessary. We can find examples in the following low-dimensional cases.
\end{rem}
 Next we construct formal 1-parameter deformations of low-dimensional JJ algebras in the classification up to dimension 5.

For an $m$-dimensional JJ algebra $\mathfrak{J}_{m}$, spanned by
$\{e_{1},\cdots,e_m\}$, define $e^{i,j}_{k}\in\mathrm{S}^{2}(\mathfrak{J}_{m},\mathfrak{J}_{m})$ by $e^{i,j}_{k}:\ (e_i,e_j)\mapsto e_k$, $1\leq i,j,k\leq m$.
To get  formal deformations, we need to consider 2-cohomology. Here we consider $\J=\J_{1,2}\oplus\F$ as an example. Our computation consists of the following steps:

(1) determine a basis of the space $\mathrm{Z}^{2}(\J,\J)$,

(2) determine a basis of the space $\mathrm{B}^{2}(\J,\J)$,

(3) determine a basis of the quotient space $\mathbb{H}$,

(4) extend the infinitesimal deformations  defined by the representative cocycles of a basis of $\mathbb{H}$.

\begin{exa}
As an example, we carry out the computation for the algebra  $\J=\J_{1,2}\oplus\F$ with basis $\{e_1,e_2,e_3\}$ and the nontrivial product $e_1^2=e_2$.
\smallskip

(1) Suppose that $\varphi\in \mathrm{S}^{2}(\J,\J)$ such that $\varphi(e_i,e_j)=\sum_{k=1}^{3}a^{i,j}_ke_k$, where $a^{i,j}_k\in\F$ for $1\leq i\leq j\leq k \leq3$. For $\mathrm{d}\varphi(e_i,e_j,e_k)=0$,  we have
$$\varphi(e_i,e_j\cdot e_k)+\varphi(e_j,e_k\cdot e_i)+\varphi(e_k,e_i\cdot e_j)+e_i\cdot\varphi(e_j,e_k)+e_j\cdot\varphi(e_k,e_i)+e_k\cdot\varphi(e_i,e_j)=0$$
for $1\leq i\leq j\leq k\leq 3$.
Equating the coefficients of $e_1$, $e_2$, $e_3$ in $\mathrm{d}\varphi(e_i,e_j,e_k)$, we get the following relations:

(i) $a^{1,2}_1=a^{1,2}_3=a^{2,2}_1=a^{2,2}_2=a^{2,2}_3=a^{2,3}_1=a^{2,3}_3=a^{3,3}_1=0$;

(ii) $a^{1,2}_2=-a^{1,1}_1$, $a^{2,3}_2=-2a^{1,3}_1$.\\
We define the following cocycles according to the above relations:

$\varphi_1=e^{1,3}_1-2e^{2,3}_2,\ \varphi_2=e^{1,3}_3,\ \varphi_3=e^{3,3}_2,\ \varphi_4=e^{3,3}_3,\
\varphi_5=e^{1,1}_1-e^{1,2}_2,\ \varphi_6=e^{1,1}_2,$

$\varphi_7=e^{1,1}_3,\ \varphi_8=e^{1,3}_2.$\\
It is easy to check that $\{\varphi_1,\ldots,\varphi_8\}$ forms a basis of $\mathrm{Z}^{2}(\J,\J)$.
\smallskip

(2) Define  $e^{j}_{i}\in \mathrm{S}^{1}(\J,\J)$ by $e^{j}_{i}: e_i\mapsto e_j$, $1\leq i,j\leq 3$. By definition, we have
 $$\mathrm{d}e^{j}_{i}(e_k,e_l)=\varphi(e_k\cdot e_l)-e_k\cdot\varphi(e_l)-e_l\cdot\varphi(e_k),\ 1\leq k,l\leq 3.$$
By a direct computation, we have
$$\mathrm{B}^{2}(\J,\J)=\mathrm{span}\{\mathrm{d}e^{j}_{i}\mid 1\leq i,j\leq 3\}=\{\varphi_4,\ldots,\varphi_8\}.$$

(3) It is straightforward to check that
$$\{\overline{\varphi}_1,\ldots,\overline{\varphi}_4\}$$
forms a basis of $\mathbb{H}$, where $\overline{\varphi}_i$ denotes the cohomology class represented by the cocycle $\varphi_i$. Thus $\mathrm{dim}\ \mathbb{H}=4$. These four cocycles give us four non-equivalent infinitesimal deformations
by Proposition 4.10, which define JJ algebras parameterized by the factor space $\F[\![t]\!]/(t^{2})$.
\smallskip

(4) To extend these four infinitesimal 1-parameter deformations,  we need to compute the Massey squares in Ch.2.
We get that only $[\varphi_2,\varphi_2]$, $[\varphi_3,\varphi_3]$ are coboundaries, which means only these two infinitesimal deformations can be extended.

(i) From $\mathrm{d}(-2e^{2,3}_{3})=-\frac{1}{2}[\varphi_2,\varphi_2]$, we can extend this infinitesimal deformation to the deformation of order 2, parameterized by $\F[\![t]\!]/(t^{3})$, with the multiplication:
$$\mu_t(x,y)=x\cdot y +\varphi_2(x,y)t-2e^{2,3}_{3}(x,y)t^{2}.$$
In order to extend it to order 3, we need to compute $[\varphi_2, -2e^{2,3}_{3}+r]$ for any $r\in \mathrm{Z}^{2}(\J,\J)$. By a direct computation, we find that $[\varphi_2, -2e^{2,3}_{3}+r]$ is not a coboundary. This means this deformation can not be extended further.

(ii) From
$[\varphi_3,\varphi_3]=0$,
we get $\varphi_3$ defines a real deformation. Moreover, we get that  this 1-parameter deformation is isomorphic to $\J_{1,3}$ for any non-zero value of $t$. An isomorphism can be given by the change of basis:
$$e'_1=e_1,\ e'_2=e_2,\ e'_3=\frac{1}{\sqrt{t}}e_3.$$
\end{exa}
\medskip

Similarly, by computing 2-cohomology with coefficients in the adjoint module  for nontrivial JJ algebras up to dimension 5,
we obtain the results in the Table below:
$$
\begin{array}{c|c|c|c|c|c|c|c|c}
\hline
  \mathrm{Dimension} & 2 & 3 & 3 & 4 & 4 & 4 & 4 & 4\\
\hline
  \mathrm{Algebra} & \mathfrak{J}_{1,2} & \mathfrak{J}_{1,2}\bigoplus \F & \mathfrak{J}_{1,3} & \mathfrak{J}_{1,2}\bigoplus \F^{2} & \mathfrak{J}_{1,3}\bigoplus \F & \mathfrak{J}_{1,2}^{2} & \mathfrak{J}_{1,4} & \mathfrak{J}_{2,4}\\
\hline
  \mathrm{dim}\ \mathbb{H} & 0 & 4 & 2 & 15 & 10 & 2 & 4 &8 \\
\hline
\end{array}
$$
$$
\begin{array}{c|c|c|c|c|c}
\hline
  5 & 5 &  5 & 5 & 5 & 5 \\
\hline
   \J_{1,2}\bigoplus \F^{3} & \J_{1,3}\bigoplus \F^{2}  & \J_{1,2}^{2}\bigoplus \F & \J_{1,4}\bigoplus \F & \J_{2,4}\bigoplus \F & \J_{1,2}\bigoplus \J_{1,3} \\
\hline
   36 & 27 & 12 & 16 & 22 &  8   \\
\hline
\end{array}
$$
$$
\begin{array}{c|c|c|c|c|c|c|c}
\hline
  5 & 5 & 5 &  5 & 5 & 5 & 5 & 5\\
\hline
  \J_{1,5} & \J_{2,5} & \J_{3,5}  & \J_{4,5}& \J_{5,5} & \J_{6,5} & \J_{7,5} & \J_{8,5}\\
\hline
  6 & 10 & 8 & 4 & 6 & 13 &  20  & 1\\
\hline
\end{array}
$$
The above 2-cohomology spaces can be spanned by the following representative cocycles:

$\J_{1,2}\oplus \F$: $e^{1,3}_1-2e^{2,3}_2,\ e^{1,3}_3,\ e^{3,3}_2,\ e^{3,3}_3,$

$\J_{1,3}$: $e^{1,3}_1-2e^{2,3}_2+2e^{3,3}_{3},\ e^{1,3}_3-2e^{3,3}_1,$

$\J_{1,2}\oplus \F^{2}$: $e^{1,3}_1-2e^{2,3}_2,\ e^{1,4}_1-2e^{2,4}_2,\
e^{1,3}_3,\ e^{1,3}_4,\ e^{1,4}_3,\ e^{1,4}_4,\ e^{3,3}_2,\ e^{3,3}_3,\ e^{3,3}_4,\ e^{3,4}_2,\ e^{3,4}_3,\ e^{3,4}_4,\ e^{4,4}_2,\ e^{4,4}_3,\ e^{4,4}_4,$

$\J_{1,3}\oplus \F$: $e^{1,3}_1-2e^{2,3}_2+2e^{3,3}_3,\ e^{1,4}_1-2e^{2,4}_2+e^{3,4}_3,\
e^{1,3}_3-2e^{3,3}_1,\ e^{1,4}_3-e^{3,4}_1,\ e^{1,1}_4,\ e^{1,3}_4,\ e^{1,4}_4,\ e^{3,4}_4,\ e^{4,4}_2,\ e^{4,4}_4,$

$\J_{1,2}^{2}$: $e^{1,3}_1-2e^{2,3}_2,\ e^{1,3}_3-2e^{1,4}_4,$

$\J_{1,4}$: $e^{1,4}_2-2e^{2,3}_2,\ e^{1,4}_4-2e^{2,3}_4,\ 2e^{3,3}_3-e^{3,4}_4,\ e^{3,3}_2,$

$\J_{2,4}$:
$e^{1,3}_1-2e^{2,3}_2+2e^{3,4}_4,\ e^{1,4}_1-2e^{2,4}_2+4e^{4,4}_4,\ e^{1,3}_3-e^{3,4}_1,\ e^{1,3}_4-2e^{3,3}_1,\ e^{1,4}_3-2e^{4,4}_1,\ e^{1,4}_4-e^{3,4}_1,\ 2e^{3,3}_3-e^{3,4}_4,\ e^{3,4}_3-2e^{4,4}_4$,

$\J_{1,2}\oplus \F^{3}$: $e^{1,4}_1-2e^{2,4}_2,\ e^{1,3}_1-2e^{2,3}_2,\ e^{1,5}_1-2e^{2,5}_2,\
e^{1,i_{1}}_{i_2},\ e^{j_1,j_2}_j,$
where $3\leq i_1,i_2\leq 5,\ 3\leq j_1\leq j_2\leq 5,\ 2\leq j\leq 5$.

$\J_{1,3}\oplus \F^{2}$:
$e^{1,3}_1-2e^{2,3}_2+2e^{3,3}_3,\ e^{1,4}_1-2e^{2,4}_2+e^{3,4}_3,\ e^{1,5}_1-2e^{2,5}_2+e^{3,5}_3,\
e^{1,3}_3-2e^{3,3}_1,\ e^{1,4}_3-e^{3,4}_1,\ e^{1,5}_3-e^{3,5}_1,\
e^{1,1}_{j_1},\ e^{1,3}_{j_1},\ e^{1,j_1}_{j_2},\ e^{3,j_1}_{j_2},\ e^{j_1,j_2}_i,$
where $j_1,j_2=4,5,\  i=2,4,5$.

$\J_{1,2}^{2}\oplus \F$: $e^{1,3}_1-2e^{2,3}_2,\ e^{1,3}_3-2e^{1,4}_4,\ e^{1,5}_1-2e^{2,5}_2,\ e^{3,5}_3-2e^{4,5}_4,\
e^{1,3}_5,\ e^{1,5}_4,\ e^{1,5}_5,\ e^{3,5}_2,\ e^{3,5}_5,\ e^{5,5}_2,\ e^{5,5}_4,\ e^{5,5}_5,$

$\J_{1,4}\oplus \F$:
$e^{1,5}_1-2e^{2,5}_2-e^{4,5}_{4},\ e^{1,4}_2-2e^{2,3}_2,\ e^{1,4}_4-2e^{2,3}_4,\ e^{1,4}_5-2e^{2,3}_5,\
e^{1,5}_3-2e^{2,5}_4,\ e^{3,4}_4-2e^{3,3}_3,\ e^{3,5}_3-e^{4,5}_4,\
e^{1,5}_2,\ e^{1,5}_5,\ e^{3,3}_2,\ e^{3,3}_5,\ e^{3,5}_2,\ e^{3,5}_5,\ e^{5,5}_2,\ e^{5,5}_4,\ e^{5,5}_5,$

$\J_{2,4}\oplus \F$: $e^{1,3}_1-2e^{2,3}_2+2e^{3,4}_{4},\ e^{1,4}_1-2e^{2,4}_2+4e^{4,4}_4,\ e^{1,5}_1-2e^{2,5}_2+2e^{4,5}_4,\ e^{1,3}_3-e^{3,4}_1,\
e^{1,3}_4-2e^{3,3}_1,\ e^{1,4}_3-2e^{4,4}_1,\
e^{1,4}_4-e^{3,4}_1,\ e^{1,5}_3-e^{4,5}_1,\ e^{1,5}_4-e^{3,5}_1,\ e^{3,4}_4-2e^{3,3}_3,\ e^{3,4}_3-2e^{4,4}_4,\ e^{3,5}_3-e^{4,5}_4,\ e^{1,1}_5,\ e^{1,3}_5,\ e^{1,4}_5,\ e^{1,5}_5,\ e^{3,3}_5,\ e^{3,5}_5,\ e^{4,4}_5,\ e^{4,5}_5,\ e^{5,5}_2,\ e^{5,5}_5,$

$\J_{1,2}\oplus \J_{1,3}$: $e^{1,3}_3-2e^{1,4}_4+e^{1,5}_{5},\ e^{3,5}_3-2e^{4,5}_4+2e^{5,5}_5,\ e^{1,3}_1-2e^{2,3}_2,\ e^{1,3}_5-e^{1,5}_3,\
e^{1,5}_1-2e^{2,5}_2,\ e^{3,5}_5-2e^{5,5}_3,\ e^{3,3}_2,\ e^{3,5}_2,$

$\J_{1,5}$: $e^{3,5}_2-2e^{1,4}_2,\ e^{3,5}_4-2e^{1,4}_4,\ e^{3,5}_5-2e^{1,4}_5,\ e^{1,5}_2-2e^{2,3}_2,\
e^{1,5}_4-2e^{2,3}_4,\ e^{1,5}_5-2e^{2,3}_5,$

$\J_{2,5}$: $e^{1,3}_1-2e^{2,3}_2-e^{3,4}_4,\ e^{1,3}_3-2e^{1,5}_5+4e^{2,4}_5,\ e^{1,4}_1-2e^{2,4}_2-2e^{4,4}_4,\ e^{3,4}_3+4e^{4,4}_4-2e^{4,5}_5,\
e^{1,3}_4-2e^{2,3}_5,\ e^{1,4}_3-e^{3,4}_4,\
e^{1,4}_4-2e^{2,4}_5,\ e^{1,4}_2,\ e^{3,4}_2,\ e^{4,4}_2,$

$\J_{3,5}$: $e^{1,3}_1-2e^{2,3}_2+2e^{3,4}_1-2e^{3,4}_3-e^{3,4}_4-2e^{3,5}_2-12e^{4,4}_4+6e^{4,5}_5,\ e^{1,3}_3+2e^{3,3}_1-2e^{3,3}_4-3e^{3,4}_4+2e^{3,5}_5,\ e^{1,5}_2-2e^{2,4}_2-e^{3,3}_1-e^{3,4}_1+e^{3,4}_3+2e^{4,4}_4-e^{4,5}_5,\
e^{3,3}_3-e^{3,4}_1+e^{3,4}_3+e^{3,4}_4+e^{3,5}_2-e^{3,5}_5+6e^{4,4}_4-3e^{4,5}_5,\
e^{1,3}_4-2e^{2,3}_5-2e^{3,3}_4,\
e^{1,5}_5-2e^{2,4}_5-e^{3,3}_4-2e^{3,4}_4+e^{3,5}_5,\ e^{4,4}_2,\ e^{4,4}_5,$

$\J_{4,5}$: $e^{1,3}_1-2e^{2,3}_2+e^{3,5}_{5}-2e^{5,5}_3,\ e^{1,3}_3-2e^{1,4}_4+e^{1,5}_5+2e^{5,5}_1,\ e^{1,3}_5-e^{1,5}_3+e^{3,5}_1,\ e^{1,5}_1-2e^{2,5}_2+e^{3,5}_3-2e^{4,5}_4+2e^{5,5}_5,$

$\mathfrak{J}_{5,5}$:
$e^{1,3}_1-2e^{2,3}_2-2e^{3,5}_3-4e^{4,5}_2+4e^{4,5}_4-4e^{5,5}_5,\ e^{1,5}_1-2e^{2,5}_2+e^{3,5}_3-2e^{4,5}_4+2e^{5,5}_5,\ e^{1,3}_3-2e^{1,4}_4+e^{1,5}_5+2e^{3,5}_1,\
e^{1,3}_5+2e^{1,4}_2-2e^{1,4}_4+2e^{1,5}_5,\ e^{3,5}_5+2e^{4,5}_2-2e^{4,5}_4+4e^{5,5}_5,\ e^{5,5}_2,$

$\mathfrak{J}_{6,5}$:
$e^{1,4}_1-2e^{2,4}_2+2e^{4,4}_4-e^{4,5}_5,\ e^{1,4}_3-2e^{2,4}_5+4e^{4,4}_3,\ e^{1,4}_4-2e^{2,4}_2+4e^{4,4}_4,\ e^{1,5}_2-2e^{2,3}_2+2e^{4,5}_2,\
e^{1,5}_5-2e^{2,3}_5+2e^{4,5}_5,\ e^{3,5}_5-2e^{3,3}_3,\ e^{3,5}_2-2e^{3,3}_4,\ e^{3,4}_3-e^{4,5}_5,\ e^{3,4}_4-e^{4,5}_2,\ e^{3,3}_2,\ e^{3,3}_5,\ e^{4,4}_2,\ e^{4,4}_5,$

$\J_{7,5}$:
$e^{1,3}_1-2e^{2,3}_2+2e^{3,3}_{3}+2e^{4,5}_3,\ e^{1,4}_1-2e^{2,4}_2+e^{3,4}_3+2e^{4,5}_5,\ e^{1,5}_1-2e^{2,5}_2+e^{3,5}_3+4e^{5,5}_5,\ e^{1,3}_4-e^{3,5}_1,\
e^{1,3}_5-e^{3,4}_1,\ e^{1,4}_3-e^{3,4}_1,\ e^{1,4}_4-e^{4,5}_1,\  e^{1,5}_3-e^{3,5}_1,\ \ e^{1,5}_5-e^{4,5}_1,\ e^{3,4}_4-e^{4,5}_3,\ e^{3,5}_5-e^{4,5}_3,\ e^{1,3}_3-2e^{3,3}_1,\
e^{1,4}_5-2e^{4,4}_1,\ e^{1,5}_4-2e^{5,5}_1,\ e^{3,5}_3-2e^{3,3}_4,
\ e^{3,4}_3-2e^{3,3}_5,\ e^{3,4}_5-2e^{4,4}_3,\ e^{3,5}_4-2e^{5,5}_3,\ e^{4,5}_5-2e^{4,4}_4,\ e^{4,5}_4-2e^{5,5}_5,$

$\J_{8,5}$: $e^{3,4}_3-2e^{4,4}_4+e^{4,5}_5+2e^{5,5}_3$.\\
Among these cocycles, we obtain that, by computing the brackets, only the following cocycles define extendible infinitesimal deformations, which are all jump deformations.

$\J_{1,2}\bigoplus \F$: $e^{3,3}_2$,

$\J_{1,3}\bigoplus \F$: $e^{1,3}_1-2e^{2,3}_2+2e^{3,3}_3,\ e^{1,3}_3-2e^{3,3}_1$,

$\J_{1,2}\bigoplus \F^{2}$: $e^{1,3}_4,\ e^{1,4}_3,\ e^{3,3}_2,\ e^{3,3}_4,\ e^{3,4}_2,\ e^{4,4}_2,\ e^{4,4}_3$,

$\J_{1,3}\bigoplus \F$: $e^{1,1}_4,\ e^{1,3}_4,\ e^{4,4}_2$,

$\J_{1,4}$: $e^{3,3}_2$,

$\J_{1,2}\oplus \F^{3}$: $e^{1,i_{1}}_{i_2},\ e^{j_1,j_2}_j,$
 where $3\leq i_1\neq i_2\leq 5,\ 3\leq j_1\leq j_2\leq 5,\ 2\leq j\leq 5,\ j\neq j_1, j_2$,

$\J_{1,3}\oplus \F^{2}$: $e^{1,1}_{j_1},\ e^{1,3}_{j_1},\ e^{1,j_1}_{j_2},\ e^{3,j_1}_{j_2},\ e^{j_1,j_1}_2,\ e^{j_1,j_1}_{j_2},\ e^{4,5}_2,$
where $j_1,j_2=4,5$, $j_1\neq j_2$,

$\J_{1,2}^{2}\oplus \F$: $e^{1,3}_5,\ e^{1,5}_4,\ e^{3,5}_2,\ e^{5,5}_2,\ e^{5,5}_4$,

$\J_{1,4}\oplus \F$: $e^{1,4}_5-2e^{2,3}_5,\ e^{1,5}_3-2e^{2,5}_4,\ e^{1,5}_2,\ e^{3,3}_2,\ e^{3,3}_5,\ e^{3,5}_2,\ e^{5,5}_2,\ e^{5,5}_4$,

$\J_{2,4}\oplus \F$: $e^{1,1}_5,\ e^{1,3}_5,\ e^{1,4}_5,\ e^{3,3}_5,\ e^{4,4}_5,\ e^{5,5}_2$,

$\J_{1,2}\oplus \J_{1,3}$: $e^{3,3}_2,\ e^{3,5}_2$,

$\J_{1,5}$: $e^{3,5}_2-2e^{1,4}_2,\ e^{1,5}_4-2e^{2,3}_4$,

$\J_{2,5}$: $e^{1,3}_4-2e^{2,3}_5,\ e^{1,4}_2,\ e^{3,4}_2,\ e^{4,4}_2$,

$\J_{3,5}$: $e^{1,3}_4-2e^{2,3}_5-2e^{3,3}_4,\ e^{4,4}_2,\ e^{4,4}_5$,

$\mathfrak{J}_{5,5}$: $e^{5,5}_2$,

$\mathfrak{J}_{6,5}$: $e^{1,4}_3-2e^{2,4}_5+4e^{4,4}_3,\ e^{3,5}_2-2e^{3,3}_4,\ e^{3,3}_2,\ e^{3,3}_5,\ e^{4,4}_2,\ e^{4,4}_5$.\\
\medskip

Finally, we get the picture of 1-parameter real deformations as follows:

in dimension 3,
$$\xymatrix{
       \mathfrak{J}_{1,2}\oplus \F \ar[d]^{}          &             \\
   \mathfrak{J}_{1,3}   &           }$$

in dimension 4,
$$
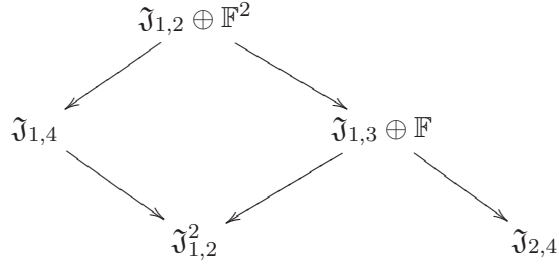
\xymatrix{
                &        \mathfrak{J}_{1,2}\oplus \F^{2}\ar[dr]\ar[dl] &  &   \\
  \mathfrak{J}_{1,4}\ar[dr]  &   &  \mathfrak{J}_{1,3}\oplus \F\ar[dl]\ar[dr]&  \\
 & \mathfrak{J}_{1,2}^{2}& & \mathfrak{J}_{2,4} }
$$

in dimension 5,
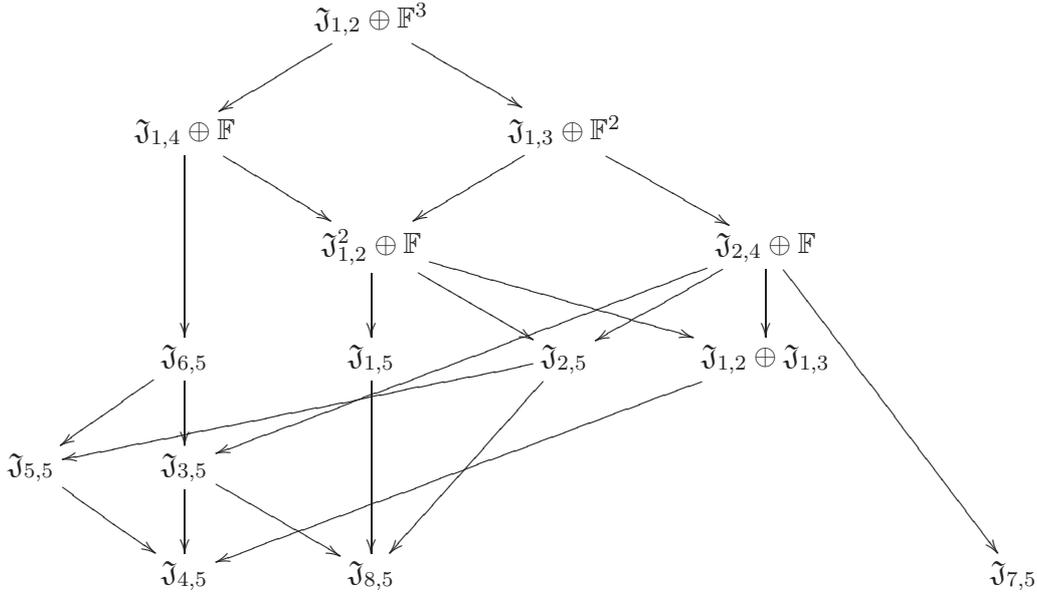
$$\xymatrix{
              &   &   \mathfrak{J}_{1,2}\oplus \F^{3}\ar[dr]\ar[dl] &  &  & &\\
  &\mathfrak{J}_{1,4}\oplus \F \ar[dr]\ar[dd]   & &  \mathfrak{J}_{1,3}\oplus \F^{2}\ar[dl]\ar[dr]&  &&\\
      &     & \mathfrak{J}_{1,2}^{2}\oplus \F\ar[d]\ar[dr]\ar[drr]    &  &  \mathfrak{J}_{2,4}\oplus \F\ar[d]\ar[dl]\ar[ddlll]\ar[dddrr] & &\\
       &  \mathfrak{J}_{6,5}\ar[dl] \ar[d] &   \J_{1,5}\ar[dd]& \J_{2,5}\ar[ddl]\ar[dlll] &\J_{1,2}\oplus \J_{1,3}\ar[ddlll] & &\\
        \mathfrak{J}_{5,5}\ar[dr]  &  \J_{3,5}\ar[d]\ar[dr]&      & &&\\
          &   \J_{4,5} & \J_{8,5} & &  & &\J_{7,5}  }
$$
where the arrows show jump deformations.
\begin{rem}
The jump deformation of the algebra $\J_{1,4}$ was also obtained in \cite{B5} as an example. However, there is a mistake there in the computation of cohomology.
\end{rem}

\section{Formal versal deformations}
\begin{defn}\label{v} \cite{exma2}
A formal deformation $\eta$ of a JJ algebra $\J$ with base $A$ is called \emph{versal} (or \emph{miniversal}) if

(1) for any formal deformation $\lambda$ of $\J$ with a base $B$, there exists a homomorphism
$\varphi:A\rightarrow B$ such that the deformation $\lambda$ is equivalent to the push-out $\varphi_{\ast}\eta$;

(2) in the notations of $(1)$, if $A$ satisfies the condition $\mathfrak{m}_{A}^{2}=0$, then  $\varphi$ is unique.
\end{defn}
Given an algebra, a main question in the study of deformations is whether there exists a versal object which can encode all deformations of this algebra. The answer is positive in the case $\mathrm{dim}\ \mathbb{H}<\infty$, which was proved by
Schlessinger's general theory \cite[Theorem 2.11]{Sch}. The method  to   construct versal deformations was worked out for Lie algebras \cite{exma2,Fuch} and Leibniz algebras \cite{Lei}. In this section, we generalize this method to the category of JJ algebras. At the end, we construct versal deformations in dimension 3.
From now on, we suppose that $\J$ is a JJ algebra with $\mathrm{dim}\ \mathbb{H}<\infty$.

\subsection{Universal infinitesimal deformation}
This Section is based on the Lie algebra case \cite{exma2, Fuch}.

We start our construction with the infinitesimal level. According to Definition \ref{v} (2), we try to choose an infinitesimal deformation which has the universality property as defined below.
\begin{defn}
An infinitesimal deformation $\eta$ of a JJ algebra $\J$ with base $A$ is called \emph{universal} if
for any infinitesimal deformation $\lambda$ of $\J$ with a finite-dimensional base $B$, there exists a unique homomorphism
$\varphi:A\rightarrow B$ such that the deformation $\lambda$ is equivalent to the push-out $\varphi_{\ast}\eta$.
\end{defn}
Next, we give a construction of a universal infinitesimal deformation, which is parallel to the Lie case \cite{Fuch} and Leibniz case \cite{Lei}.
Denote by $\mathbb{H}'$ the dual space of $\mathbb{H}$. Consider the algebra $A=\F\oplus\mathbb{H}'$ with the second summand being an ideal with zero multiplication. It is easy to see that $\mathbb{H}'$ is the maximal ideal of $A$ and its square is zero.
For  $\alpha\in \mathbb{H}$, we denote by $\mu(\alpha)$ a cocycle representing the cohomology class $\alpha$.
From the isomorphism
$$A\otimes \J=(\F\otimes\J)\oplus(\mathbb{H}'\otimes\J)=\J\oplus \mathrm{Hom}(\mathbb{H},\J),$$
we define a JJ algebra structure on $A\otimes \J$ with the formula
$$(x_1,\varphi_1)\cdot(x_2,\varphi_2)=(x_1x_2,\psi),$$
where the mapping $\psi:\mathbb{H}\rightarrow\J$ is given by
$$\psi(\alpha)=\mu(\alpha)(x_1,x_2)+x_1\varphi_2(\alpha)+x_2\varphi_1(\alpha),$$
for $x_1,x_2\in\J$, $\varphi_1,\varphi_2\in\mathrm{Hom}(\mathbb{H},\J)$, $\alpha\in\mathbb{H}$.
It is straightforward to check that the above multiplication defines an infinitesimal deformation with base $A$ (the Jacobi identity follows from $\mathrm{d}\mu(\alpha)=0$).

\begin{lem}
Up to isomorphism, the above infinitesimal deformation does not depend on the choice of $\mu$. Denote this deformation by $\eta_{\J}$.
\end{lem}
\begin{proof}
Let $\mu'(\alpha)$ be another cocycle representing the cohomology class $\alpha$. Then there exists $\gamma(\alpha)\in \mathrm{S}^{1}(\J,\J)$ such that $\mu'(\alpha)=\mu(\alpha)+\mathrm{d}\gamma(\alpha)$.
Define a linear endomorphism $\rho$ of the space $A\otimes \J=\J\oplus \mathrm{Hom}(\mathbb{H},\J)$ by the formula
$$\rho(x,\varphi)=(x,\psi),\quad \psi(\alpha)=\varphi(\alpha)+\gamma(\alpha)(x),$$
where $x\in\J$, $\varphi\in\mathrm{Hom}(\mathbb{H},\J)$ and  $\alpha\in\mathbb{H}$.
Clearly, $\rho$ is a linear automorphism and the inverse of $\rho$ is given by replacing $\gamma$ with
$-\gamma$ in the formula. In order to prove that $\rho$ preserves the multiplication, it is sufficient to show that
for any $x_1, x_2\in\J$, $\varphi_1 ,\varphi_2\in \mathrm{Hom}(\mathbb{H},\J)$, and $\alpha\in\mathbb{H}$, one has
\begin{eqnarray*}
&&\mu(\alpha)(x_1,x_2)+x_1\varphi_2(\alpha)+x_2\varphi_1(\alpha)+\gamma(\alpha)(x_1\cdot x_2) \\
&&  \qquad\qquad  \qquad=\mu'(\alpha)(x_1,x_2)+x_1\cdot(\varphi_2(\alpha)+\gamma(\alpha)x_2)
+x_2\cdot(\varphi_1(\alpha)+\gamma(\alpha)x_1).
\end{eqnarray*}
The proof follows from  the equality $\mu'(\alpha)=\mu(\alpha)+\mathrm{d}\gamma(\alpha)$ and the definition of $\mathrm{d}$.
\end{proof}

\begin{rem}
Suppose $\{h_i\}_{i=1}^{n}$ is a basis of $\mathbb{H}$ and $\{h'_i\}_{i=1}^{n}$ is the dual basis. Let $\mu(h_i)=\mu_i\in \mathrm{Z}^{2}(\J,\J)$. Under the identification $A\otimes \J=\J\oplus \mathrm{Hom}(\mathbb{H},\J)$, an element $(x,\varphi)\in \J\oplus\mathrm{Hom}(\mathbb{H},\J)$
corresponds to $1\otimes x+\sum_{i=1}^{n}h'_i\otimes\varphi(h_i)\in A\otimes \J$. Then for $(x_1,\varphi_1), (x_2, \varphi_2)\in \J\oplus\mathrm{Hom}(\mathbb{H},\J)$, their multiplication $(x_1x_2,\psi)$ corresponds to
$$1\otimes x_1x_2+\sum_{i=1}^{n}h'_i\otimes(\mu_i(x_1,x_2)+ x_1 \varphi_2(h_i)+x_2\varphi_1(h_i)).$$
In particular, for $x_1,x_2\in\J$ we have
\begin{equation}\label{in}
(1\otimes x_1\cdot 1\otimes x_2)_{\eta_{\J}}=1\otimes x_1x_2+\sum_{i=1}^{n}h'_i\otimes\mu_i(x_1,x_2).
\end{equation}
\end{rem}

\begin{lem}\label{uni}
Let $\lambda$ be an infinitesimal deformation of the JJ algebra $\J$ with the finite-dimensional base $A$. Take $\xi\in\mathfrak{m}_{A}'$, or, equivalently, $\xi\in A'$ and $\xi(1) = 0$. For $x_1,x_2\in\J$, define a cochain
$\alpha_{\lambda,\xi}\in\mathrm{S}^{2}(\J,\J)$ satisfying
$$\alpha_{\lambda,\xi}(x_1,x_2)=(\xi\otimes \mathrm{id}_{\J})(1\otimes x_1\cdot1\otimes x_2)_\lambda \in \F\otimes \J=\J, \quad x_1,x_2\in\J.$$

(1) The cochain $\alpha_{\lambda,\xi}$ is a cocycle. Moreover, define two mappings:
$\alpha_{\lambda}:\mathfrak{m}_{A}'\rightarrow \mathrm{Z}^{2}(\J,\J)$ and $a_{\lambda}:\mathfrak{m}_{A}'\rightarrow \mathbb{H}$ such that $\alpha_{\lambda}(\xi)=\alpha_{\lambda,\xi}$ and $a_{\lambda}(\xi)$ is the cohomology class of $\alpha_{\lambda,\xi}$.
\smallskip

(2) Let $\lambda'$ be another infinitesimal deformation of $\J$ with base $A$. Then $\lambda$ and $\lambda'$ are equivalent if and only if $a_{\lambda}=a_{\lambda'}$.
\end{lem}
\begin{proof}
(1) For $x_1,x_2,x_3\in \J$, by the definition of the differential $\mathrm{d}$,
\begin{equation}\label{1}
\mathrm{d}\alpha_{\lambda,\xi}(x_1,x_2,x_3)=\sum_{\circlearrowleft x_1,x_2,x_3}\alpha_{\lambda,\xi}(x_1,x_2 x_3)+x_1\alpha_{\lambda,\xi}(x_2,x_3).
\end{equation}
By Definition \ref{e},
\begin{equation}\label{2}
(1\otimes x_2\cdot1\otimes x_3)_{\lambda}=1\otimes x_2x_3+\sum_i m_i\otimes y_i
\end{equation}
where $m_i\in\mathfrak{m}_{A}$, $y_i\in\J$. Hence
\begin{eqnarray*}\label{3}
(\xi\otimes\mathrm{id})(1\otimes x_1\cdot(1\otimes x_2\cdot1\otimes x_3)_{\lambda})_{\lambda}&=&(\xi\otimes\mathrm{id})(1\otimes x_1\cdot1\otimes x_2x_3)_{\lambda}\\
&+&(\xi\otimes\mathrm{id})\sum_i m_i(1\otimes x_1\cdot1\otimes y_i)_{\lambda}\\
&=&\alpha_{\lambda,\xi}(x_1,x_2x_3)+(\xi\otimes\mathrm{id})\sum_i m_i(1\otimes x_1\cdot1\otimes y_i)_{\lambda}
\end{eqnarray*}
For the second term, one has
$$m_i(1\otimes x_1\cdot1\otimes y_i)_{\lambda}=m_i(1\otimes x_1y_i+h),$$
where $h\in \mathfrak{m}_{A}\otimes\J$. Since $\mathfrak{m}_{A}^{2}=0$, we have
$m_i(1\otimes x_1\cdot1\otimes y_i)_{\lambda}=m_i\otimes x_1y_i$. Hence
\begin{eqnarray*}
(\xi\otimes\mathrm{id})(1\otimes x_1\cdot(1\otimes x_2\cdot1\otimes x_3)_{\lambda})_{\lambda}&=&\alpha_{\lambda,\xi}(x_1,x_2x_3)+(\xi\otimes\mathrm{id})\sum_i m_i\otimes x_1y_i \\
   &=&\alpha_{\lambda,\xi}(x_1,x_2x_3)+\sum_i \xi(m_i)( x_1y_i) \\
   &=&\alpha_{\lambda,\xi}(x_1,x_2x_3)+x_1\left(\sum_i \xi(m_i)y_i\right)  \\
   &=&  \alpha_{\lambda,\xi}(x_1,x_2x_3)+x_1\left((\xi\otimes\mathrm{id})\sum_i m_i\otimes y_i\right)
\end{eqnarray*}
From Eq. (\ref{2}), we have
\begin{eqnarray*}
(\xi\otimes\mathrm{id})\sum_i m_i\otimes y_i&=&(\xi\otimes\mathrm{id})((1\otimes x_2\cdot1\otimes x_3)_{\lambda}-1\otimes x_2x_3)\\
   &=&(\xi\otimes\mathrm{id})(1\otimes x_2\cdot1\otimes x_3)_{\lambda}\\
   &=&\alpha_{\lambda,\xi}(x_2,x_3)
\end{eqnarray*}
In conclusion,
$$(\xi\otimes\mathrm{id})(1\otimes x_1\cdot(1\otimes x_2\cdot1\otimes x_3)_{\lambda})_{\lambda}=
\alpha_{\lambda,\xi}(x_1,x_2x_3)+x_1\alpha_{\lambda,\xi}(x_2,x_3).$$
From Eq. (\ref{1}),
\begin{eqnarray*}
\mathrm{d}\alpha_{\lambda,\xi}(x_1,x_2,x_3)&=&\sum_{\circlearrowleft_{x_1,x_2,x_3}}\alpha_{\lambda,\xi}(x_1,x_2 x_3)+x_1\alpha_{\lambda,\xi}(x_2,x_3) \\
&=&(\xi\otimes\mathrm{id}) \sum_{\circlearrowleft_{x_1,x_2,x_3}} (1\otimes x_1\cdot(1\otimes x_2\cdot1\otimes x_3)_{\lambda})_{\lambda}
\end{eqnarray*}
From the Jacobi identity for $(,)_\lambda$, $\mathrm{d}\alpha_{\lambda,\xi}=0$.

(2)
Suppose that $\rho$ is an $A$-linear automorphism of $A\otimes\J$  satisfying $(\varepsilon_{A}\otimes \mathrm{id})\circ\rho=\varepsilon_{A}\otimes \mathrm{id}$.
In fact, $\rho$ is determined by the value $\rho(1\otimes x)$ for $x\in\J$ by  $A$-linearity.
We have the following isomorphisms
$$ A\otimes\J=(\F\oplus\mathfrak{m}_{A})\otimes\J=(\F\otimes\J)\oplus(\mathfrak{m}_{A}\otimes\J)=\J\oplus(\mathfrak{m}_{A}\otimes\J).
$$
Choose linear mappings
$a_{\rho}:\J\rightarrow\J$ and $b_{\rho}:\J\rightarrow \mathfrak{m}_{A}\otimes\J$ such that
$\rho(1\otimes x)=a_{\rho}(1\otimes x)+b_{\rho}(1\otimes x)$ for any $x\in\J$.
By the compatibility $(\varepsilon_{A}\otimes \mathrm{id})\circ\rho=\varepsilon_{A}\otimes \mathrm{id}$, we have $a_{\rho}=\mathrm{id}_{\J}$.  We use the isomorphism
\begin{align*}
  \mathrm{Hom}(\J,\mathfrak{m}_{A}\otimes\J) &\longrightarrow\mathfrak{m}_{A}\otimes \mathrm{C}^{1}(\J,\J), \\
  b_{\rho}&\mapsto\sum_{i=1}^{r}m_i\otimes \varphi_{i},
\end{align*}
where $\{m_i\}_{i=1}^{r}$ is a basis of $\mathfrak{m}_{A}$, $\{m'_i\}_{i=1}^{r}$ is the dual basis and  $\varphi_i=(m'_i\otimes \mathrm{id})\circ b_{\rho}$.
We have
\begin{equation}\label{iso}
\rho(1\otimes x)=1\otimes x+\sum_{i=1}^{r}m_i\otimes \varphi_{i}(x),\quad x\in\J.
\end{equation}
Consider $\rho((1\otimes x_1\cdot1\otimes x_2)_{\lambda})$ and $(\rho(1\otimes x_1)\cdot\rho(1\otimes x_2))_{\lambda'}$ for $x_1,x_2\in\J$. From Eq. (\ref{iso}), we have
$$\rho((1\otimes x_1\cdot1\otimes x_2)_{\lambda})=1\otimes x_1x_2+\sum_{i=1}^{r}m_i\otimes\varphi_{i}(x_1x_2)+\sum_{i=1}^{r}m_i\otimes\alpha_{\lambda,m'_i}(x_1,x_2)$$
and
\begin{eqnarray*}
(\rho(1\otimes x_1)\cdot\rho(1\otimes x_2))_{\lambda'}&=&1\otimes x_1x_2+\sum_{i=1}^{r}m_i\otimes\alpha_{\lambda',m'_i}(x_1,x_2)+\sum_{i=1}^{r}m_i\otimes (x_1\varphi_i(x_2)) \\
  && +\sum_{i=1}^{r}m_i\otimes (x_2\varphi_i(x_1)).
\end{eqnarray*}
 Hence $\rho$ establishes an isomorphism between the JJ algebras  $\lambda$ and $\lambda'$ if only and if $\alpha_{\lambda',m'_i}-\alpha_{\lambda,m'_i}=\mathrm{d}\varphi_i$ for $1\leq i\leq r$, that is, $a_{\lambda}=a_{\lambda'}$.
\end{proof}

\begin{thm}
$\eta_{\J}$ is universal infinitesimal deformation of $\J$.
\end{thm}
\begin{proof}
For an infinitesimal deformation $\lambda$ of $\J$ with a finite-dimensional base $A$, consider the mapping
$$\varphi=\mathrm{id}\oplus a'_{\lambda}:\F\oplus\mathbb{H}'\rightarrow \F\oplus \mathfrak{m}_A=A,$$
where $a'_{\lambda}$ is the dual mapping of $a_{\lambda}$, that is, $m'(a'_{\lambda}(h'))=h'(a_{\lambda}(m'))$ for $h\in \mathbb{H}$, $h'\in \mathbb{H}'$, $m\in \mathfrak{m}_{A}$ and $m'\in \mathfrak{m'}_{A}$.
Let $\{h_i\}_{i=1}^{n}$ be  a basis of $\mathbb{H}$, $\{m_i\}_{i=1}^{r}$  a basis of $\mathfrak{m}_{A}$ and  $\{h'_i\}_{i=1}^{n}$, $\{m'_i\}_{i=1}^{r}$ be the corresponding dual basis. Then we have
\begin{equation}\label{e1}
a_{\lambda}(m'_i)=\sum^{n}_{j=1}h'_{j}(a_{\lambda}(m'_i))h_j,\quad 1\leq i\leq r.
\end{equation}
From Eqs. (\ref{po}), (\ref{in}) and (\ref{e1}), for $1\leq i\leq r$, $x_1,x_2\in\J$,
\begin{eqnarray*}
\alpha_{\varphi_{\ast}\eta_{\J}}(m'_i)(x_1,x_2)&=&(m'_i\otimes \mathrm{id})(1\otimes x_1\cdot1\otimes x_2)_{\varphi_{\ast}\eta_{\J}} \\
&=&(m'_i\otimes \mathrm{id})(1\otimes x_1x_2+\sum_{j=1}^{n}\varphi(h'_j)\otimes\mu(h_j)(x_1,x_2))  \\
  &=&\sum_{j=1}^{n}m'_i(a'_{\lambda}(h'_j))\otimes\mu(h_j)(x_1,x_2) \\
   &=&\sum_{j=1}^{n}h'_j(a_{\lambda}(m'_i))\otimes\mu(h_j)(x_1,x_2) \\
   &=&\mu\left(\sum_{j=1}^{n}h'_j(a_{\lambda}(m'_i))h_j\right)(x_1,x_2)\\
   &=&\mu(a_{\lambda}(m'_i))(x_1,x_2).
\end{eqnarray*}
Then $a_{\varphi_{\ast}\eta_{\J}}=a_{\lambda}$. By Lemma \ref{uni} (2), $\lambda$ is equivalent to $\varphi_{\ast}\eta_{\J}$ and  $\varphi$ is unique.
\end{proof}

\subsection{Extension of the universal infinitesimal deformation}
The universal infinitesimal deformation, constructed in Section 5.1, gives the construction of versal deformation of order 1. Using the same method as for Lie and Leibniz cases, we extend it to higher order to get a versal formal deformation. Here we introduce the extensions without proofs, for more details and examples, see \cite{Fuch, Lei,exam1,exma2}.

Given a JJ algebra $\J$, let $\{\varphi_i\}^{n}_{i=1}$  be cocylces whose cohomology classes form a basis of $\mathbb{H}$. By Section 5.1, a versal deformation of order 1, parameterized by  $\F[\![t_1,\ldots,t_n]\!]/(\mathfrak{m}^{2})=\mathrm{span}_{\F}\{1,t_1,\cdots,t_n\}$, can be written as
\begin{equation}\label{re}
\mu_t=\mu_0+t_1\varphi_1+\cdots+t_n\varphi_n,
\end{equation}
where  $\mu_0$ is the original multiplication of $\J$.  Now let us try to extend it to a deformation of order 2 parameterized by $\F[\![t_1,\ldots,t_n]\!]/I$ where $I$ is an ideal containing $\mathfrak{m}^{3}$.
This deformation can be written as
$$\mu_t=\mu_0+\sum^{n}_{i=1}t_i\varphi_i+\sum t_it_j \varphi_{ij},$$
where $\varphi_{ij}\in \mathrm{S}^{2}(\J,\J)$. The relations on the base are given by the conditions
$$-2\sum \mathrm{d}\varphi_{ij}t_{i}t_{j}\equiv\sum[\varphi_{i},\varphi_{j}]t_{i}t_{j}\ \mathrm{mod}\ I.$$
This means that the right-hand side (which is always a three-cocycle) must be coboundary. So the ideal $I$ is generated by the polynomials, obtained by Massey products of the right-hand side and $\mathfrak{m}^{3}$. For $\varphi_{ij}$ one can choose any 2-cochain satisfying the above condition.
\begin{rem}
In particular, if all brackets are zero,
then the multiplication (\ref{re}) defines a JJ algebra itself. This means we get a versal deformation without any further construction.
\end{rem}

\subsection{Examples}

Now we carry out the construction of formal versal deformation for 3-dimensional JJ algebras  $\J_{1,2}\oplus \F$ and $\J_{1,3}$. First, we get a versal deformation of order 1 with the help of the universal infinitesimal deformation, parameterized by $\F[\![t_1,\ldots,t_n]\!]/(\mathfrak{m}^{2})$, where $\mathfrak{m}$ is the maximal ideal of $\F[\![t_1,\ldots,t_n]\!]$, generated by $t_1,\ldots, t_n$. Then we extend the deformation to the next order by computing all  Massey squares, including mixed ones.

\begin{exa}
(Formal versal deformation of $\J_{1,2}\oplus \F$)
From Section 4, the 2-cohomology is spanned by the following representative cocycles:
$$\varphi_1=e^{1,3}_1-2e^{2,3}_2,\ \varphi_2=e^{1,3}_3,\ \varphi_3=e^{3,3}_2,\ \varphi_4=e^{3,3}_3.$$
Thus, the universal infinitesimal deformation is given by
$$(1\otimes e_i \cdot 1\otimes e_j)_{\eta}=1\otimes e_ie_j+\sum_{k=1}^{4}t^{k}\otimes\varphi_k(e_i,e_j),\quad 1\leq i,j\leq 3,$$
parameterized by $A=\mathrm{span}_{\F}\{1,t_1,t_2,t_3,t_4\}$. Now, we try to extend this deformation. By computing all Massey products, except for $[\varphi_3,\varphi_3]=0$ (it gives us a real deformation considered as 1-parameter deformation), we get that every bracket gives a nontrivial cochain. These cochains, except for $[\varphi_2,\varphi_2]$ and $[\varphi_2,\varphi_3]$, are not coboundaries. Among these cochains, six of them are linearly independent.
They give us second order relations on the parameter space:
$$t^{2}_1=t^{2}_4=t_1t_2=t_1t_4=t_2t_4=t_3t_4-2t_1t_3=0.$$
From $\mathrm{d}(-2e^{2,3}_3)=-\frac{1}{2}[\varphi_2,\varphi_2]$ and $\mathrm{d}(-2e^{3,3}_1)=-\frac{1}{2}([\varphi_2,\varphi_3]+[\varphi_3,\varphi_2])$,
we get that the versal  deformation $(,)_{v}$ is defined
as follows:
\begin{eqnarray*}
  &&(e_1^{2})_{v}=e_2,\\
  &&(e_1e_3)_{v}=t_1e_1+t_2e_3, \\
  &&(e_2e_3)_{v}=-2t_1e_2-2t_{2}^{2}e_3, \\
  &&(e^{2}_3)_{v}=t_3e_2+t_4e_3-2t_2t_3e_1,
\end{eqnarray*}
parameterized by the factor space
$$\F[\![t_1,t_2,t_3,t_4]\!]/I,$$
where $I$ is generated by $t^{2}_1,t^{2}_4,t_1t_2,t_1t_4,t_2t_4,t_3t_4-2t_1t_3$ and $\mathfrak{m}^{3}$.
\end{exa}

To extend it to the next order, we have to compute Massey 3-products. But among those, only $[\varphi_2,\varphi_2]$ and $[\varphi_2,\varphi_3]$ are coboundaries. This means only the Massey 3-product $<\overline{\varphi}_2,\overline{\varphi}_2, \overline{\varphi}_3>$ is defined, but the others are  not coboundaries. So the versal 2-order deformation is a versal deformation as well.
\smallskip
The only 1-parameter real deformation is $(\cdot,\cdot)_t=(\cdot,\cdot)+t\varphi_3$.

\begin{exa}(Versal deformation of $\J_{1,3}$)
From Section 4, the 2-cohomology is spanned by the following two representative cocycles:
$$\varphi_1=e^{1,3}_1-2e^{2,3}_2+2e^{3,3}_{3},\ \varphi_2=e^{1,3}_3-2e^{3,3}_1.$$
Thus, the universal infinitesimal  deformation is given by
$$(1\otimes e_i\cdot 1\otimes e_j)_{\eta}=1\otimes e_ie_j+\sum_{k=1}^{2}t^{k}\otimes\varphi_k(e_i,e_j),\quad 1\leq i,j\leq 3,$$
parameterized by $A=\mathrm{span}_{\F}\{1,t_1,t_2\}$. Now, we try to extend this deformation. By computing all Massey products, we get three nontrivial cochains:
$$[\varphi_1,\varphi_1],\quad [\varphi_1,\varphi_2],\quad [\varphi_2,\varphi_2].$$
These three cochains are all not coboundaries and they are linearly independent. We get three relations on the parameter space:
$$t^{2}_1=t_1t_2=t^{2}_{2}=0.$$
As no higher order Massay product is defined, we get the versal deformation $(,)_v$ defined by
\begin{eqnarray*}
&&(e_1^{2})_v=e_{2} \\
&&(e_1e_3)_v=t_1e_1+t_2e_3, \\
&&(e_2e_3)_v=-2t_1e_2,\\
&&(e_3^{2})_v=e_2+2t_1e_3-2t_2e_1,
\end{eqnarray*}
and it is parameterized by the factor space
$$\F[\![t_1,t_2]\!]/(t^{2}_1,t^{2}_2,t_1t_2).$$
\end{exa}
\begin{rem}
In order to extend this deformation to third order, we have to compute Massey 3-products. But it turns out that none of them is defined, so this second order deformation is also a versal deformation. The 2-cohomology is not trivial, however this algebra is rigid.
\end{rem}
\begin{conclusion}
When we extend the universal infinitesimal deformation of $\J_{1,2}\oplus \F$ and $\J_{1,3}$, we get nontrivial Massey products. These nontrivial products give us the relations on the base.
We have shown in   that the JJ algebra $\J_{1,3}$ is formally rigid, while $\J_{1,2}\oplus \F$ has one real 1-parameter deformation which is isomorphic to $\J_{1,3}$ (see Ex. 3).
\end{conclusion}

\bigskip

\small\noindent \textbf{Acknowledgment}\\
The author would like to thank the Doctoral School of Physics, Univeristy of P\'ecs for support, and Prof. Alice Fialowski for giving the problem and regular useful discussions. All data generated or analysed during this study are included in this published article (and its supplementary information files). 
\bigskip

\small\noindent \textbf{Conflicts of interests}\\
The author declares that there are no conflicts of interests.

\end{document}